\documentclass[11pt,a4paper]{article}
\usepackage[T1]{fontenc}
\usepackage{lmodern,amsmath,amssymb,amsthm,mathtools,microtype}
\usepackage[margin=27mm]{geometry}
\usepackage[colorlinks=true,linkcolor=blue,citecolor=blue,urlcolor=blue]{hyperref}
\newtheorem{theorem}{Theorem}[section]
\newtheorem{lemma}[theorem]{Lemma}
\newtheorem{proposition}[theorem]{Proposition}

\theoremstyle{remark}
\newcommand{\E}{\mathbb E}\newcommand{\PP}{\mathbb P}
\newcommand{\R}{\mathbb R}\newcommand{\ind}{\mathbf1}
\newcommand{\ip}[2]{\langle #1,#2\rangle}
\newcommand{\RD}{\operatorname{RD}}\newcommand{\KL}{\operatorname{KL}}
\newcommand{\Law}{\mathcal L}\newcommand{\clip}{\operatorname{clip}}
\newcommand{\cx}{\preceq_{\rm cx}}\newcommand{\st}{\preceq_{\rm st}}
\newcommand{\RR}{\mathcal R}\newcommand{\eps}{\varepsilon}
\title{An invariant variational proof for canonical processes\\with independent symmetric log-concave tails}
\author{{Witold Bednorz, Rafal Martynek and Rafal Meller}
\footnote{{\bf Subject classification:} 60G15, 60G17}
\footnote{{\bf Keywords and phrases:} Canonical Processes,  Invariant Method}
\footnote{Research partially supported by  Grant UMO-2022/47/B/ST1/02114}
\footnote{Institute of Mathematics, University of Warsaw, Banacha 2, 02-097 Warszawa, Poland}}
\date{}
\begin{document}
\maketitle
\begin{abstract}
We prove a prescribed-index-law comparison for canonical processes with independent symmetric coordinates having log-concave tails, with no doubling assumption. The proof starts with an invariant variational argument for capped exponentials, lifts the orbit estimate by type-class replication, and proves the reverse bound by clipped codes. We then give the tail-slope expansion and prove both its cost comparison and its random-variable comparison, including the countable-expansion limits. Thus the general-tail step is proved here rather than left as a transfer assertion. The classical exponential chaining theorem and standard facts of convex analysis, log-concavity and finite-dimensional transport remain explicit external inputs. This manuscript merges the two supplied drafts; the general-tail transfer arguments are attributed to Hu--Wang--Wu. No new comparison theorem is claimed.
\end{abstract}
\tableofcontents
\clearpage

\section{The theorem and the structure of the proof}
Let $Y=(Y_1,\ldots,Y_d)$ have independent, symmetric, nonzero coordinates. Assume
\[
N_i(u)=-\log\PP(|Y_i|\ge u),\qquad u\ge0,
\]
are closed, convex and nondecreasing, with $N_i(0)=0$ and extended values allowed. Put
$q_i(p)=\sup\{u\ge0:N_i(u)\le p\}$ and normalize $q_i(1)=1$. Coordinate rescaling is harmless: replace $Y_i$ by $Y_i/q_i(1)$ and every index coordinate $x_i$ by $q_i(1)x_i$. Zero coordinates, if any, can simply be deleted.
The inverse $q_i$ is nonnegative, nondecreasing and concave, and
\begin{equation}\label{quantile}
|Y_i|\overset{d}=q_i(E),\qquad q_i(ap)\le a q_i(p)\quad(a\ge1),
\end{equation}
where $E$ is a mean-one exponential. These follow by inversion of the convex tail potential, including an atom at the upper endpoint.
Let $b_i=\lim_{p\to\infty}q_i(p)\in[1,\infty]$. If $b_i>1$, define
\begin{equation}\label{potential}
V_i(u)=\begin{cases}
u^2,&|u|\le1,\\
2N_i(|u|)-1,&1<|u|\le b_i,\\
+\infty,&|u|>b_i.
\end{cases}
\end{equation}
If $b_i=1$, take $V_i(u)=u^2$ on $[-1,1]$ and infinity outside. These potentials are even and convex: when $b_i>1$, $N_i(1)=1$ and $N'_{i,+}(1)\ge1$. For $0<s<b_i$ put $\sigma_i(s)=V'_{i,+}(s)$ and
\begin{equation}\label{intrinsic}
\Psi_i(r)=\int_0^{b_i}\sigma_i(s)\ind_{\{\sigma_i(s)<r\}}\,ds,\qquad r\ge0.
\end{equation}
The strict inequality fixes the values at slope thresholds. In particular $\Psi_i(r)=r^2/4$ for $0\le r\le2$, and $\Psi_i(r)\ge\chi(r/2)$, where $\chi(r)=\min(1,r^2)$.

For a centered integrable source $A$ and a finitely supported probability law $\mu$ on $\R^d$, set
\begin{equation}\label{pairing}
B_A(\mu)=\sup_{X\sim\mu}\E\ip{X}{A},\qquad
S_A(T)=\E\max_{x\in T}\ip{x}{A}.
\end{equation}
The supremum in $B_A$ is over all couplings with the prescribed source law. For finite $T$,
\begin{equation}\label{supmu}
S_A(T)=\sup_{\mu\in\mathcal P(T)}B_A(\mu),
\end{equation}
since a measurable maximizing index gives the reverse of the immediate inequality.
For any nondecreasing costs $d_i:[0,\infty)\to[0,\infty]$ with $d_i(0)=0$, write
\begin{align}
D_t^d(x,y)&=\sum_i d_i(|x_i-y_i|/t),\nonumber\\
\RD_\mu^d(t)&=\inf_{X,X'\sim\mu}\{I(X;X')+\E D_t^d(X,X')\},\label{RD}\\
\RR_d(\mu)&=\int_0^\infty\RD_\mu^d(t)\,dt.\nonumber
\end{align}
Both marginals are prescribed. All logarithms are natural.

\begin{theorem}\label{main}
There is a universal $C$ such that
\[
C^{-1}\RR_\Psi(\mu)\le B_Y(\mu)\le C\RR_\Psi(\mu)
\]
for every finitely supported $\mu$. Consequently,
$S_Y(T)\asymp\sup_{\mu\in\mathcal P(T)}\RR_\Psi(\mu)$ for every finite $T$.
In particular this applies to independent symmetric coordinates with log-concave densities.
\end{theorem}
For the last assertion, the survival function of $|Y_i|$ on $[0,\infty)$ is log-concave by Pr\'ekopa's theorem. Symmetric signs are also permitted by the tail formulation, although they do not have densities.

The invariant capped-exponential argument comes from the supplied draft \cite{invariant}; the second supplied note \cite{extension} explains the final reduction but invokes the transfer results of \cite{HWW}. Sections~\ref{area}--\ref{finish} below include proofs of those transfer results, following \cite{HWW}, Section 4. We use as standard inputs Pr\'ekopa's theorem \cite{Prekopa}, the exponential chaining theorem \cite{LT}, Birkhoff's finite assignment theorem \cite{Birkhoff}, the martingale characterization of convex order \cite{Strassen}, and elementary finite-dimensional convex duality. We do not invoke the uniform convex-order truncation theorem of \cite{HWW}.

\section{Couplings, information, and limiting operations}\label{prelim}
We collect the facts needed to keep both index marginals fixed. If an observation $O$ is generated from $X\sim\mu$, generate $X'$ independently from the posterior law of $X$ given $O$. Then
\begin{equation}\label{posterior}
X'\sim\mu,\quad (X',O)\overset d=(X,O),\quad I(X;X')\le I(X;O).
\end{equation}
This is data processing. For observations $U,V$ conditionally independent given $X$,
$I(X;U,V)\le I(X;U)+I(X;V)$, by the entropy chain rule.
We also use the relative-entropy inequality
\begin{equation}\label{entropy}
\E_P f\le\KL(P\Vert Q)+\log\E_Q e^f.
\end{equation}
It follows by comparing $P$ with the exponential tilt of $Q$.

\begin{lemma}\label{calculus}
The costs in \eqref{RD} obey the following rules:
\begin{align}
\RR_{d(a\cdot)}(\mu)&=a\RR_d(\mu),\quad a>0,\label{scale}\\
\RR_d(\mu)&\le\RR_{cd}(\mu)\le c\RR_d(\mu),\quad c\ge1,\nonumber\\
\RR_{d+e}(\mu)&\le4\{\RR_d(\mu)+\RR_e(\mu)\},\label{costsum}\\
\RR_d(\mu)&\le8\{\RR_d(a_\#\mu)+\RR_d(v_\#\mu)\}
\quad\text{if }x=a(x)+v(x).\label{decomp}
\end{align}
Injective pullback of the index and distortion leaves the functional unchanged.
\end{lemma}
\begin{proof}
The first two assertions follow by a change of scale and by comparing objectives. Monotonicity of each coordinate cost gives
\begin{equation}\label{triangle}
D_{2t}^d(x,z)\le D_t^d(x,y)+D_t^d(y,z).
\end{equation}
For \eqref{costsum}, generate conditionally independent optimal self-coupling observations $U,V$ for $d,e$ at scale $t$, and posterior-resample $X'$ given $(U,V)$. Formula \eqref{posterior} and \eqref{triangle} give
$\RD^{d+e}_\mu(2t)\le2\RD^d_\mu(t)+2\RD^e_\mu(t)$. Integrate.
For \eqref{decomp}, use optimal observations of $a(X)$ and $v(X)$ instead. Applying \eqref{triangle} twice gives
$\E D_{4t}^d(X,X')\le2\E D_t^d(a(X),U)+2\E D_t^d(v(X),V)$.
The information bound is the sum of the two component informations. Integration proves the claim. Injectivity gives a bijection of the finite coupling tables, preserving information and cost.
\end{proof}

We record two limit facts. If $A_m\to A$ in $L_1$ under a coupling, gluing couplings gives
\begin{equation}\label{sourceconv}
|B_{A_m}(\mu)-B_A(\mu)|\le M_\mu\E\|A_m-A\|_2,
\quad M_\mu=\max_{x\in\operatorname{supp}\mu}\|x\|_2.
\end{equation}
If finite-label cost matrices $c_m$ increase entrywise to $c$, their optimized information-plus-cost values increase to the value for $c$. Indeed, take a convergent subsequence of minimizers on the compact coupling polytope. For each fixed $k$, lower semicontinuity bounds the objective with $c_k$ at the limit by the limit of the minima. Let $k\to\infty$. Testing any fixed coupling gives the opposite inequality. This includes infinite entries, with $0\cdot\infty=0$. Monotone convergence then applies to the integrated values.

We use $A\cx B$ for convex order. The martingale characterization supplies a coupling with $\E[B\mid A]=A$, and thus
\begin{equation}\label{cxpair}
A\cx B\quad\Longrightarrow\quad B_A(\mu)\le B_B(\mu).
\end{equation}
To check this, lift any coupling of $X,A$ through the martingale kernel. Independent coordinatewise convex orders give a vector convex order by successive conditional expectations. For symmetric scalars, magnitude stochastic domination implies convex order: sign-averaging any convex function makes it even and nondecreasing on $[0,\infty)$. Also $(\E|A|)\eps\cx A$ for symmetric $A$, by conditional Jensen given its sign.

\section{Capped exponentials and the invariant variational bound}
Let $Z_i=\eps_iE_i$ be independent symmetric mean-one exponentials, and
$\xi_i^\beta=\eps_i\min(E_i,\beta_i)$, with $1\le\beta_i\le\infty$.
Their coordinate costs are
\begin{equation}\label{capcost}
\psi_\beta(r)=\begin{cases}r^2/4,&0\le r\le2,\\2\beta-1,&r>2.\end{cases}
\end{equation}
The second branch is infinite when $\beta=\infty$. Write $\RR_\beta$ for this family.

\subsection{Classical exponential input}
For $p\ge2$,
\begin{equation}\label{expmoment}
\|\ip{h}{Z}\|_p\asymp\sqrt p\|h\|_2+p\|h\|_\infty,
\qquad \|\ip{h}{Z}\|_1\ge c\|h\|_2.
\end{equation}
The upper estimate follows from $\E e^{sZ_i}=(1-s^2)^{-1}$. A largest-coordinate conditioning gives the $p\|h\|_\infty$ lower estimate; the representation $Z_i\overset d=\sqrt{2E_i}g_i$ gives the Gaussian term by Jensen for $p\ge2$. Second and fourth moments give the $L_1$ bound by interpolation.
The exponential chaining theorem \cite{LT} gives maps $\pi_j:T\to T$ for every finite $T$, with $\pi_0$ constant, $|\pi_j(T)|\le e^{2^j}$, eventually $\pi_j=\mathrm{id}$, and
\begin{equation}\label{chain}
\sup_{x\in T}\sum_{j\ge0}\bigl(2^{j/2}\|x-\pi_jx\|_2+2^j\|x-\pi_jx\|_\infty\bigr)
\le C S_Z(T).
\end{equation}
It also implies the following comparison. If $A$ is centered and
$\|\ip{h}{A}\|_p\le a\|\ip{h}{Z}\|_p$ for all $h,p\ge1$, then
\begin{equation}\label{expcomparison}
S_A(T)\le Ca S_Z(T).
\end{equation}
For clarity, one can obtain this directly by chaining along the maps in \eqref{chain}: a union bound and Markov's inequality at a sufficiently large multiple of $2^j$ control all level-$j$ increments by their exponential moment size. Integration of the union bound and \eqref{expmoment} give \eqref{expcomparison}. The constant root term has expectation zero.

\begin{lemma}[Conditional tails]\label{condtail}
If $S$ has an integrable density proportional to $e^{-|s|}h(s)$ with $h$ log-concave, and $m=\E|S|$, then
\[
m\le|\E S|+2,\qquad \E(|S|-t)_+\le Cm e^{-ct/m}.
\]
If $|\E S|\le b$ and $b\ge1$, then
\begin{equation}\label{clipmean}
|\E(S-\clip_{Kb}S)|\le\epsilon_K(1+|\E S|),\qquad
\epsilon_K=Ce^{-cK}.
\end{equation}
\end{lemma}
\begin{proof}
If the support is one-sided, $m=|\E S|$. Otherwise $0$ is interior to its support. Convexity of $-\log h$ implies that $h$ is nonincreasing on at least one of the two half-lines when directed away from zero. On that side the density satisfies $f(s+t)\le e^{-t}f(s)$, so the corresponding one-sided first moment is at most one. Thus
$m=|\E S|+2\min(\E S_+,\E S_-)\le|\E S|+2$.
The survival function of $S$ is log-concave by Pr\'ekopa. Markov gives
$\PP(S\ge4m)\le1/4$ and $\PP(S\ge-4m)\ge3/4$.
Extrapolation of its log-survival between these two points gives exponential decay beyond $4m$. Repeat for $-S$, enlarge constants below $4m$, and integrate. Finally $m\le3b$ and $m\le2(1+|\E S|)$ prove \eqref{clipmean}.
\end{proof}

\begin{lemma}[Residual contraction]\label{residual}
Let $U=\E[Z\mid\mathcal F]$. Let $\eta$ have independent coordinates of laws $\delta_iZ_i'$, where $Z_i'$ are symmetric mean-one exponentials and the independent selectors satisfy $\E\delta_i\le q$. No other independence between $Z,\eta,\mathcal F$ is required. If $R=\E[\eta\mid\mathcal F]$ and
$|R_i|\le\epsilon(1+|U_i|)$, $0<\epsilon\le1$, then
\begin{equation}\label{resbound}
\|\ip{h}{R}\|_p\le C(\sqrt\epsilon+\sqrt q)\|\ip{h}{Z}\|_p,
\quad S_R(T)\le C(\sqrt\epsilon+\sqrt q)S_Z(T).
\end{equation}
\end{lemma}
\begin{proof}
For $p\ge2$, let $I$ consist of the $k=\min(d,\lceil p/\epsilon\rceil)$ largest $|h_i|$. Conditional Jensen and the coordinate bound yield
\[
\|\ip{h_I}{R}\|_p\le\epsilon\|h_I\|_1+
\epsilon\Big\|\sum_{i\in I}|h_i||Z_i|\Big\|_p
\le C\sqrt{\epsilon p}\|h\|_2+C\epsilon p\|h\|_\infty.
\]
Here the nonnegative exponential sum has $L_p$ norm at most
$\sum a_i+C\sqrt p\|a\|_2+Cp\|a\|_\infty$, by its moment generating function.
Similarly, $\E e^{s\delta_iZ_i'}=1+q_i s^2/(1-s^2)$ gives
$\|\ip{a}{\eta}\|_p\le C\sqrt{qp}\|a\|_2+Cp\|a\|_\infty$.
On $I^c$, $\|h_{I^c}\|_\infty\le\sqrt{\epsilon/p}\|h\|_2$, so Jensen gives a bound
$C(\sqrt q+\sqrt\epsilon)\sqrt p\|h\|_2$ for its contribution.
Combine with \eqref{expmoment}. For $1\le p<2$, use the result at $2$ and the $L_1$ lower bound. Finally $R$ is centered, so \eqref{expcomparison} applies.
\end{proof}

Let $G$ be a finite group of coordinate permutations preserving finite caps $\beta$.
Write $F_A(u)=S_A(Gu)$ and
\[
\|a\|_{\beta,1}=\sum_i\beta_i|a_i|,\qquad
J_\beta(u)=\inf_v\{\|u-v\|_{\beta,1}+F_Z(v)\}.
\]
\begin{proposition}\label{invariantbound}
$F_{\xi^\beta}(u)\le J_\beta(u)\le C F_{\xi^\beta}(u)$ universally.
\end{proposition}
\begin{proof}
The first inequality is the triangle inequality and sign contraction. For the other, smooth the reference functional:
\[
F_\tau(v)=\tau\E\log\sum_{g\in G}e^{\ip{gv}{Z}/\tau},\quad
F_Z(v)\le F_\tau(v)\le F_Z(v)+\tau\log|G|.
\]
Since $F_\tau\ge\tau\log|G|$, the continuous objective
$\|u-v\|_{\beta,1}+F_\tau(v)$ is coercive. At a minimizer $v_\tau$, put $w=\nabla F_\tau(v_\tau)$. Optimality gives
\begin{equation}\label{optimal}
|w_i|\le\beta_i,\qquad \ip{u-v_\tau}{w}=\|u-v_\tau\|_{\beta,1}.
\end{equation}
Differentiation under the expectation is justified by the integrable gradient bound $\|Z\|_2$.
Choose a random $L\in G$ conditionally on $Z=z$ with probabilities
\[
\kappa_g(z)=\frac{e^{\ip{gv_\tau}{z}/\tau}}{\sum_h e^{\ip{hv_\tau}{z}/\tau}}.
\]
The pair $(L,Z)$ has the law of $(hL,hZ)$ for every $h\in G$.
Hence $L$ is uniform, $w=\E[L^{-1}Z]$, and
\[
\E[Z\mid L]=Lw.
\]
The Gibbs entropy identity is
$F_\tau(v_\tau)=\ip{v_\tau}{w}+\tau H(L\mid Z)$.
Together with \eqref{optimal} and $H(L\mid Z)\le\log|G|$, it implies
\begin{equation}\label{gibbs}
J_\beta(u)\le\ip{u}{w}+\tau\log|G|.
\end{equation}
Put $\zeta=Z-\clip_{K\beta}Z$, $R=\E[\zeta\mid L]$, and $r=\E[L^{-1}\zeta]$.
Equivariance gives $R=Lr$. The conditional density given $L=g$ is proportional to
$e^{-\|z\|_1}\kappa_g(z)$. Since $\log\kappa_g$ is affine minus log-sum-exp, it is concave. Pr\'ekopa then gives each conditional marginal the form in Lemma~\ref{condtail}.
Thus, with $U=Lw$,
\[
|R_i|\le\epsilon_K(1+|U_i|),\quad |r_i|\le2\epsilon_K\beta_i.
\]
Memorylessness makes $\zeta_i$ independent sparse exponentials with selector probabilities $e^{-K\beta_i}\le e^{-K}$. Lemma~\ref{residual} gives, for every $v$,
\[
\ip{v}{r}\le\max_g\ip{gv}{r}
=\E\max_g\ip{gv}{R}\le\theta_K F_Z(v),\qquad \theta_K\le Ce^{-cK}.
\]
Enlarge $\theta_K$ so $|r_i|\le\theta_K\beta_i$. Decomposing $u=(u-v)+v$ and infimizing gives $\ip{u}{r}\le\theta_KJ_\beta(u)$.
Also
\[
\ip{u}{w}=\E\ip{Lu}{\clip_{K\beta}Z}+\ip{u}{r}
\le K F_{\xi^\beta}(u)+\theta_KJ_\beta(u),
\]
using magnitude contraction $\min(E_i,K\beta_i)\le K\min(E_i,\beta_i)$.
Fix a universal $K$ with $\theta_K\le1/2$, use \eqref{gibbs}, and let $\tau\downarrow0$. The bound is uniform in $G$.
\end{proof}

\section{Haar entropy and replication}\label{haar}
For $T=Gu$, let $\lambda$ be uniform on $T$. Transitivity and cost invariance make
\[
Z_t=\sum_y\lambda(y)e^{-D_t^\beta(x,y)}
\]
independent of $x$. Symmetry gives the same column sums. Therefore
\[
Q_t(x,y)=\lambda(x)\lambda(y)e^{-D_t^\beta(x,y)}/Z_t
\]
is a self-coupling. For any finite-cost coupling $P$,
$I(P)+\E_P D_t^\beta=-\log Z_t+\KL(P\Vert Q_t)$. Consequently
\begin{equation}\label{haaridentity}
\RD_\lambda^\beta(t)=-\log Z_t,\qquad
\RR_\beta(\lambda)=\int_0^\infty-\log Z_t\,dt=:H_\beta(T).
\end{equation}
For a $G$-invariant source, $B_A(\lambda)=S_A(T)$: take a maximizing index and apply an independent uniform group element simultaneously to it and the source.

\begin{lemma}\label{uncapped}
For any probability law $\nu$ supported on a finite set $T$, $\RR_\infty(\nu)\le C S_Z(T)$.
\end{lemma}
\begin{proof}
Take \eqref{chain} and set
$a_j(x)=\tfrac12\max\{2^{-j/2}\|x-\pi_jx\|_2,\|x-\pi_jx\|_\infty\}$.
If $t>a_j(x)$, then $D_t^\infty(x,\pi_jx)\le2^j$.
Let $J_t(x)$ be the first such $j$, and observe $O=(J_t(X),\pi_{J_t(X)}X)$.
There are at most $e^{2^j}$ labels at level $j$. Assigning each weight $e^{-3\cdot2^j}$ gives total weight at most one and hence $H(O)\le3\E2^{J_t(X)}$ by nonnegativity of relative entropy. Posterior resampling and \eqref{triangle} yield
$\RD_\nu^\infty(2t)\le C\E2^{J_t(X)}$.
Put $R_* =\max_xa_0(x)=\tfrac12\max_x\|x-\pi_0x\|_2$. Since
$2^J=1+\sum_{j<J}2^j$ and $J_t(x)>j$ implies $t\le a_j(x)$,
\[
\int_0^{R_*}\E2^{J_t(X)}dt\le R_*+\E\sum_j2^ja_j(X).
\]
For $t>R_*$ a constant observation and independent copies give
$\RD_\nu^\infty(2t)\le2\E D_t^\infty(X,\pi_0X)\le2R_*^2/t^2$.
Integrate and apply \eqref{chain}. If $R_*=0$, the support is a singleton and the claim is immediate.
\end{proof}

\begin{proposition}\label{haarlower}
For any invariant orbit, including infinite caps,
$H_\beta(Gu)\le C F_{\xi^\beta}(u)$.
\end{proposition}
\begin{proof}
First take finite caps. Average a candidate $v$ over the stabilizer of $u$. Convexity and invariance show that neither term defining $J_\beta(u)$ increases. Then $gu\mapsto g(u-v)$ and $gu\mapsto gv$ are well-defined, with uniform orbit image laws.
Direct integration gives $\int_0^\infty\psi_\beta(|a|/t)dt=\beta|a|$.
The independent self-coupling therefore gives
$\RR_\beta(\lambda_{u-v})\le2\|u-v\|_{\beta,1}$.
Lemma~\ref{uncapped} and $\psi_\beta\le\psi_\infty$ give
$\RR_\beta(\lambda_v)\le C F_Z(v)$.
Use \eqref{decomp}, infimize, and apply Proposition~\ref{invariantbound}.
For infinite caps replace $\beta_i$ by $\beta_i\wedge m$; the Haar integrands increase and the sources converge in $L_1$. Monotone convergence proves the assertion.
\end{proof}

\begin{proposition}[Type-class limits]\label{types}
Let $\mu=\sum_{j=1}^m p_j\delta_{a_j}$ have positive rational weights. For multiples $N$ of their common denominator, let $T_N(\mu)$ consist of sequences with exactly $Np_j$ occurrences of $a_j$, and let $\lambda_N$ be uniform. For independent copies of any centered integrable $A$,
\begin{equation}\label{typelimits}
\frac1N S_{A^{\oplus N}}(T_N(\mu))\longrightarrow B_A(\mu),\qquad
\frac1N\RR_\beta(\lambda_N)\longrightarrow\RR_\beta(\mu),
\end{equation}
where caps on the left are repeated in each block.
\end{proposition}
\begin{proof}
For $\nu_N=N^{-1}\sum_{k=1}^N\delta_{A_k}$, the assignment theorem gives
$N^{-1}\max_{x\in T_N}\sum_k\ip{x_k}{A_k}=B_\mu(\nu_N)$.
Splitting each mass $Np_j$ into slots reduces this to the usual doubly stochastic assignment problem. Gluing couplings makes $B_\mu$ Lipschitz in $W_1$ with constant $\max_j\|a_j\|_2$. Also $\E W_1(\nu_N,\Law(A))\to0$: approximate $A$ in $L_1$ by a finite-valued vector, apply convergence of empirical frequencies to it, and bound the two approximation errors by the original $L_1$ error. This proves the first limit.

Fix a reference sequence in $T_N$. The proportion of sequences with joint empirical table $P=(P_{ij})$ is
\[
w_N(P)=\frac{\prod_i(Np_i)!\prod_j(Np_j)!}{N!\prod_{ij}(NP_{ij})!}.
\]
The margins are $p$ and all $NP_{ij}$ are integers. Uniformly,
$\log w_N(P)=-N\KL(P\Vert p\otimes p)+O_m(\log(N+1))$;
there are at most $(N+1)^{m^2}$ tables. Using \eqref{haaridentity} shows that
$N^{-1}\RD_{\lambda_N}^\beta(t)$ differs by $O_m(\log(N+1)/N)$ from the minimum of
$\KL(P\Vert p\otimes p)+\sum_{ij}P_{ij}D_t^\beta(a_i,a_j)$ over those tables.
These minima converge to the unrestricted self-coupling minimum. Indeed, delete infinite-cost edges; the remaining nonempty transportation polytope is rational. Approximate any feasible point by a rational table $Q$ of denominator $L$ divisible by the denominator of $p$. If $N=kL+R$, the table
$P_N=(kL/N)Q+(R/N)\operatorname{diag}(p)$ is admissible and tends to $Q$. Compactness proves the converse bound.
Finally the diagonal coupling bounds $N^{-1}\RD_{\lambda_N}^\beta(t)$ by $H(p)$, and for $t\ge\operatorname{diam}_\infty(\operatorname{supp}\mu)/2$ the independent coupling bounds it by
$\operatorname{diam}_2(\operatorname{supp}\mu)^2/(4t^2)$. Dominated convergence proves the integrated limit.
\end{proof}

Apply Proposition~\ref{haarlower} to the block-permutation orbit $T_N$, divide by $N$, and use \eqref{typelimits}. This gives $\RR_\beta(\mu)\le C B_{\xi^\beta}(\mu)$ for rational weights. For general weights on fixed finite support, approximate by positive rational weights. Compactness of coupling tables and lower semicontinuity give
$\RD_\mu(t)\le\liminf_n\RD_{\mu_n}(t)$; Fatou applies. Pairings converge: if $\Delta$ is the support diameter, maximal coupling of the labels gives, for every $M>0$,
\[
|B_A(\mu_n)-B_A(\mu)|\le\Delta\{M\|\mu_n-\mu\|_{\rm TV}
+\E[\|A\|_2\ind_{\{\|A\|_2>M\}}]\}.
\]
First let $n\to\infty$, then $M\to\infty$. Hence
\begin{equation}\label{caplower}
\RR_\beta(\mu)\le C B_{\xi^\beta}(\mu)
\end{equation}
for every finite law and all caps.

\section{The clipped-code upper bound}\label{upper}
We prove the other direction, following the clipped-code construction in \cite{invariant,HWW}. First suppose all caps are finite. Set
$d_i(r)=\chi(r)+\beta_i\ind_{\{r>1\}}$. Directly,
\begin{equation}\label{auxcost}
\tfrac12d_i(r)\le\psi_{\beta_i}(2r)\le2d_i(r),\qquad
\RR_d(\mu)\le4\RR_\beta(\mu).
\end{equation}
We will construct maps $x=a(x)+v(x)$ such that
\begin{equation}\label{codegoal}
\inf_{a+v=\mathrm{id}}\{\E\|a(X)\|_{\beta,1}+B_Z(v_\#\mu)\}
\le C\RR_d(\mu).
\end{equation}
Let $r_k=2^{-k}$ and choose integers $m<N$ with
$r_m\ge2\operatorname{diam}_\infty(\operatorname{supp}\mu)$.
For $m<k<N$, take minimizing self-coupling observations $U_k$ at scale $r_k$, conditionally independently given $X\sim\mu$. Set $U_m\sim\mu$ independently and $U_N=X$. Write
\[
h_k=I(X;U_k),\quad e_k=\E D_{r_k}^d(X,U_k),\quad
g_k=\E\sum_i\chi(|X_i-U_{k,i}|/r_k).
\]
At interior levels $h_k+e_k=\RD_\mu^d(r_k)$ and $g_k\le e_k$.
Define
\[
\Delta_k=\clip_{3r_k/2}(U_{k+1}-U_k),\quad
\widetilde v=U_m+\sum_{k=m}^{N-1}\Delta_k,\quad
v(x)=\E[\widetilde v\mid X=x],\quad a(x)=x-v(x).
\]

\paragraph{Discarded increments.}
In a fixed coordinate, call level $k$ bad if $|X_i-U_{k,i}|>r_k$. The two endpoints are good. An edge between good levels loses nothing on clipping. For a maximal bad run $p,\ldots,q$, telescope the unclipped increments on edges $p-1,\ldots,q$. The total discarded part on this block is at most
\[
r_{p-1}+r_{q+1}+\tfrac32\sum_{k=p-1}^q r_k\le\tfrac{17}{2}r_p.
\]
These edge blocks are disjoint, so
$|X_i-\widetilde v_i|\le\tfrac{17}{2}\sum_{k=m+1}^{N-1}r_k\ind_{\{|X_i-U_{k,i}|>r_k\}}$.
Multiplying by caps and using Jensen gives
\begin{equation}\label{discard}
\E\|a(X)\|_{\beta,1}\le\tfrac{17}{2}\sum_{k=m+1}^{N-1}r_ke_k.
\end{equation}

\paragraph{Retained increments.}
The inequality
$\min(3/2,s+t/2)^2\le\tfrac94[\chi(s)+\chi(t)]$ gives
\begin{equation}\label{increment}
\E\|\Delta_k\|_2^2\le\tfrac94r_k^2(g_k+g_{k+1}).
\end{equation}
For any finite observation $O$ and $\|w(O)\|_\infty\le s/2$,
\begin{equation}\label{expentropy}
\E\ip{w(O)}{Z}\le s I(Z;O)+\frac2s\E\|w(O)\|_2^2.
\end{equation}
Indeed, $\log\E e^{\ip{z}{Z}}\le2\|z\|_2^2$ for $\|z\|_\infty\le1/2$; apply \eqref{entropy} to each conditional source law with test function $\ip{w(o)}{Z}/s$.
Under any coupling of $X,Z$, generate the observations through the same fixed kernels, independently of $Z$ given $X$. Then
$I(Z;U_k,U_{k+1})\le h_k+h_{k+1}$.
Apply \eqref{expentropy} with $s=3r_k$, use \eqref{increment}, and sum. As $r_{k-1}=2r_k$, this gives
\[
\sum_{k=m}^{N-1}\E\ip{\Delta_k}{Z}
\le9\sum_{k=m+1}^{N-1}r_kh_k+\tfrac92\sum_{k=m+1}^{N-1}r_kg_k
+6r_NH(\mu)+\tfrac32r_mg_m.
\]
The root $U_m$ is independent of $Z$. Also
$\E[\widetilde v\mid X,Z]=v(X)$, and $v$ is the same map for every coupling. Any coupling of $v(X)$ with $Z$ lifts through the conditional law of $X$ given $v(X)$. Thus the same bound holds for $B_Z(v_\#\mu)$.
Together with \eqref{discard} and $r_mg_m=\E\|X-U_m\|_2^2/r_m$, this proves
\[
\E\|a(X)\|_{\beta,1}+B_Z(v_\#\mu)
\le13\sum_{k=m+1}^{N-1}r_k\RD_\mu^d(r_k)
+\frac{3\operatorname{diam}_2(\operatorname{supp}\mu)^2}{2r_m}+6r_NH(\mu).
\]
Since $\RD(t)$ is nonincreasing,
$\sum_{k\in\mathbb Z}r_k\RD(r_k)\le2\int_0^\infty\RD(t)dt$.
Take the infimum over decompositions and let $m\to-\infty$, $N\to\infty$ to prove \eqref{codegoal}.

Finally, $\xi^\beta\cx Z$ by sign contraction and independent products. Consequently every decomposition satisfies
\[
B_{\xi^\beta}(\mu)\le\E\|a(X)\|_{\beta,1}
+B_{\xi^\beta}(v_\#\mu)
\le\E\|a(X)\|_{\beta,1}+B_Z(v_\#\mu).
\]
Use \eqref{codegoal} and \eqref{auxcost}. For infinite caps, truncate the caps, use \eqref{sourceconv}, and monotonicity of the costs. Combined with \eqref{caplower}, this proves
\begin{theorem}[Capped comparison]\label{capped}
For every finite index law and $\beta\in[1,\infty]^d$,
$B_{\xi^\beta}(\mu)\asymp\RR_\beta(\mu)$ with universal constants.
\end{theorem}
In particular, $(1-e^{-1})\eps\cx\xi^1\cx\eps$ coordinatewise. Since $\psi_1(r)=\chi(r/2)$, Theorem~\ref{capped} and \eqref{scale} give
\begin{equation}\label{bernoulli}
B_\eps(\mu)\asymp\RR_\chi(\mu).
\end{equation}
This Bernoulli consequence will absorb the errors in the tail expansion; it is already available before the general-tail proof begins.

\section{The slope-area identity}\label{area}
We prove the convex-geometric identity needed for the transfer, following \cite{HWW}.
For even convex potentials $V_i$ as in \eqref{potential}, write
$\rho_p(h)=\sup\{\ip{h}{z}:\sum_iV_i(z_i)\le p\}$.
\begin{lemma}\label{areaid}
For $p>0$,
\begin{equation}\label{areaformula}
\rho_p(h)=\int_0^\infty\min\{p,D_t^\Psi(0,h)\}\,dt.
\end{equation}
\end{lemma}
\begin{proof}
Let $f_i=V_i^*$. With $\sigma_i=V'_{i,+}$, integration of the derivative and maximization at the point where the integrand changes sign give
\[
f_i(r)=\int_0^{b_i}(r-\sigma_i(u))_+du,\qquad r\ge0.
\]
This includes flat slopes and unattained endpoints, by monotone limits. Tonelli then gives
\[
\int_t^\infty\Psi_i(a/v)dv
=\int_0^{b_i}(a-t\sigma_i(u))_+du=tf_i(a/t).
\]
Thus $H(t):=t\sum_i f_i(|h_i|/t)=\int_t^\infty D_v^\Psi(0,h)dv$.
Fenchel's inequality gives $\rho_p(h)\le pt+H(t)$.
Conversely, the finite, nondecreasing, concave function $q\mapsto\rho_q(h)$ has a supporting line at $p$ with slope $\lambda\ge0$.
If $\lambda>0$, its supporting inequality gives
$H(\lambda)=\sup_z\{\ip{h}{z}-\lambda\sum_iV_i(z_i)\}\le\rho_p(h)-\lambda p$.
If $\lambda=0$, the same inequality bounds $H(t)\le\rho_p(h)$ and one lets $t\downarrow0$.
Hence $\rho_p(h)=\inf_{t>0}\{pt+H(t)\}$.
Finally $v\mapsto D_v^\Psi(0,h)$ is nonincreasing with an integrable quadratic tail. Splitting its integral where it crosses $p$ shows that this infimum is exactly the right side of \eqref{areaformula}. The value at a crossing point does not affect the integral.
\end{proof}

\section{Dyadic slope expansion and cost comparison}\label{costtransfer}
Set $s_j=2^{j+1}$ and $a_j=2/s_j=2^{-j}$, $j\ge0$. For each coordinate define
\[
A_{ij}=\{u\in[1,b_i):s_j\le\sigma_i(u)<2s_j\},\qquad
w_{ij}=s_j|A_{ij}|.
\]
Each nonempty bin is an interval. Keep precisely the bins with $w_{ij}\ge1$, and set $\beta_{ij}=(w_{ij}+1)/2$, allowing infinity. Let $\Xi_{i,*}$ have cap one and $\Xi_{ij}$ have cap $\beta_{ij}$, with all expanded coordinates independent. Define
\begin{equation}\label{expansion}
H_i=\Xi_{i,*}+\sum_{j:w_{ij}\ge1}a_j\Xi_{ij},\quad
(Lx)_{i,*}=x_i,\quad (Lx)_{ij}=a_jx_i.
\end{equation}
The coefficients sum to at most $3$. Since $\|\Xi_\alpha\|_p\le\|Z_1\|_p$, each series converges absolutely almost surely and in every fixed $L_p$. The core makes $L$ injective.
The induced and error costs are
\begin{equation}\label{Phi}
\Phi_i(r)=\chi(r/2)+\sum_{j:w_{ij}\ge1}\psi_{\beta_{ij}}(a_jr),\qquad
E_0(r)=\sum_{j\ge0}\chi(r/s_j).
\end{equation}

\begin{proposition}\label{costprop}
For all $i,r$,
\begin{equation}\label{pointcost}
\Psi_i(r)\le2\Phi_i(r)+2E_0(r),\qquad
\Phi_i(r/2)\le\Psi_i(r)+E_0(r/2),\qquad
\int_0^\infty E_0(r)r^{-2}dr=2.
\end{equation}
Moreover $\RR_\Psi(\mu)\asymp\RR_\Phi(\mu)$ for every finite index law.
\end{proposition}
\begin{proof}
Put $\Theta_i(r)=\chi(r/2)+\sum_jw_{ij}\ind_{\{r>s_j\}}$.
On a bin, the slope lies between $s_j$ and $2s_j$, so
$\Theta_i(r/2)\le\Psi_i(r)\le2\Theta_i(r)$.
For a kept bin,
\[
\psi_{\beta_{ij}}(a_jr)=w_{ij}\ind_{\{r>s_j\}}+(r/s_j)^2\ind_{\{r\le s_j\}}.
\]
Discarded steps have weights less than one, and the added quadratic terms are also bounded by the corresponding summands in $E_0$.
Writing $\Theta_i=\Phi_i-Q_i+A_i$ with $0\le Q_i,A_i\le E_0$ proves the first two inequalities, also at thresholds. Tonelli and
$\int_0^\infty\chi(r/s)r^{-2}dr=2/s$ give the last identity.

To control the error for a prescribed law, truncate $E_0$ at $j=J$ and pull back the Bernoulli cost through $x\mapsto(x_i/s_j)_{i,j\le J}$. Its source in original coordinates is $W_i^{(J)}=\sum_{j=0}^J\eps_{ij}/s_j$. This is centered and lies in $[-1,1]$, so $W_i^{(J)}\cx\eps_i$: a convex function lies below the chord joining its values at $-1,1$. Independence and \eqref{cxpair}, together with \eqref{bernoulli}, imply
$\RR_{E_0^{(J)}}(\mu)\le C B_\eps(\mu)$.
The pullback pairing identity used here follows by lifting a coupling through the conditional law of the expanded source given its linear image, exactly as in Section~\ref{finish} below. Increasing $J$ and using the finite-label limit in Section~\ref{prelim} gives
$\RR_{E_0}(\mu)\le C B_\eps(\mu)$.
Both $\Psi_i,\Phi_i$ dominate $\chi(r/2)$, so their functionals are at least $cB_\eps(\mu)$.
Now apply \eqref{pointcost}, scaling, and the sum rule \eqref{costsum} to absorb the error in both directions.
\end{proof}

\section{Comparison of the expanded random variables}\label{variables}
We give the remaining source comparison, following \cite{HWW}, with details so that it is not a black-box transfer.

\begin{lemma}[Moments imply convex order]\label{momentcx}
Let $Y$ be a normalized symmetric log-concave-tailed scalar, with quantile $q(1)=1$. Suppose $A$ is symmetric and, for fixed positive constants $c,C,c_0$,
\[
cq(p)\le\|A\|_p\le Cq(p)\quad(p\ge2),\qquad \E|A|\ge c_0.
\]
Then $C_1^{-1}Y\cx A\cx C_1Y$, where $C_1$ depends only on these constants.
\end{lemma}
\begin{proof}
Markov gives $\PP(|A|>eCq(p))\le e^{-p}$.
Since $q(2p)\le2q(p)$, Paley--Zygmund applied to $|A|^p$ at level $2^{-p}\E|A|^p$ gives
\[
\PP(|A|>\delta q(p))\ge e^{-\kappa p},\qquad p\ge2,
\]
for $\delta=c/4$ and a constant $\kappa\ge1$. The strict inequality is valid because $\|A\|_p/2\ge cq(p)/2>\delta q(p)$.
We show
\begin{equation}\label{stadd}
|A|\st M(1+|Y|),\qquad |Y|\st M(1+|A|)
\end{equation}
for a large fixed $M$.
For the first, at a point $y<b$ with $N(y)\ge2$, use the upper tail estimate at $p=N(y)$ and $q(N(y))=y$. At $N(y)<2$, use $q(2)\le2$ and the estimate $e^{-2}\le\PP(|Y|>y)$. At $y\ge b<\infty$, the moment upper bound gives $\|A\|_\infty\le Cb$, so the left tail vanishes for large $M$.
For the second, if $r<\delta q(2)$ choose $M\ge2\kappa$; then
$\PP(|Y|>M(1+r))\le e^{-2\kappa}\le\PP(|A|>r)$.
If $\delta q(2)\le r<\delta b$, continuity gives $r=\delta q(p)$ for some $p\ge2$.
Choose $M\ge\kappa/\delta$. Since $q(\kappa p)\le\kappa q(p)$,
$\PP(|Y|>M(1+r))\le e^{-\kappa p}\le\PP(|A|>r)$.
If $r\ge\delta b$ with finite $b$, choose $M\ge1/\delta$ and the left tail is zero. Endpoint atoms cause no problem because these comparisons use strict tails.

To remove the added one, if a symmetric $B$ satisfies $\E|B|\ge b_0>0$, then $b_0\eps\cx B$ and for any convex $f$,
\[
\E f(M\eps(1+|B|))
\le\tfrac12\E f(2M\eps)+\tfrac12\E f(2MB)
\le\E f(C_*B),\quad C_*=2M\max(1,b_0^{-1}).
\]
Here the sign is independent, and scaling a centered variable up preserves convex order. Apply this to $B=Y$ and $B=A$ in \eqref{stadd}, noting $\E|Y|\ge e^{-1}$ by \eqref{quantile}.
\end{proof}

\begin{proposition}\label{sourceprop}
For the independent coordinates of \eqref{expansion},
$C^{-1}Y_i\cx H_i\cx CY_i$ universally. Hence $B_H(\mu)\asymp B_Y(\mu)$.
\end{proposition}
\begin{proof}
Fix one coordinate and omit its index. Define
\[
Q(p)=\sup\{u\ge0:V(u)\le p\},\qquad
Q_\Phi(p)=\int_0^\infty\min\{p,\Phi(1/t)\}\,dt.
\]
The area identity and \eqref{pointcost} give
\begin{equation}\label{profiles}
Q(p)\le2Q_\Phi(p)+4,\qquad Q_\Phi(p)\le2Q(p)+2.
\end{equation}
For example $\int E_0(1/t)dt=2$; in the other direction substitute $r=1/t$ into the second inequality of \eqref{pointcost} and rescale $t$.
Both profiles are at least one when $p\ge1$, so they are comparable there.
If $b>1$, $Q(p)=q((p+1)/2)\asymp q(p)$ for $p\ge1$; if $b=1$, both equal one. The representation \eqref{quantile} also gives
\begin{equation}\label{Ymom}
e^{-1}q(p)\le\|Y\|_p\le
q(p)\bigl(1+\Gamma(p+1)/p^p\bigr)^{1/p}\le Cq(p).
\end{equation}
For the lower bound restrict to $E\ge p$; for the upper bound use $q(E)\le q(p)\max(1,E/p)$.

We prove $\|H\|_p\asymp Q_\Phi(p)$ first for finite expansions. Write
$H=\sum_jc_j\xi_j^{\beta_j}$, including the independent core with coefficient one and cap one; $c_j\ge0$ and $\sum_jc_j\le3$.
Let
\[
V_\beta(u)=\begin{cases}u^2,&|u|\le1,\\2|u|-1,&1<|u|\le\beta,\\+\infty,&|u|>\beta.\end{cases}
\]
Its slope cost is $\psi_\beta$. Lemma~\ref{areaid} identifies
\begin{equation}\label{support}
Q_\Phi(p)=\sup\left\{\sum_jc_ju_j:\sum_j V_{\beta_j}(u_j)\le p\right\}.
\end{equation}
For $f_\beta=V_\beta^*$ and $z\ge0$,
\begin{equation}\label{mgf}
\log\E e^{z\xi^\beta}\le f_\beta(8z).
\end{equation}
Indeed for $z\le1/4$ the left side is at most $2z^2$, while the right side equals $16z^2$. For $1/4\le z\le1/2$ the left side is at most $\log(4/3)$ and the right side is at least one. For $z\ge1/2$ and finite $\beta$, use
$\log\E e^{z\xi^\beta}\le\beta z\le8\beta z-2\beta+1=f_\beta(8z)$; for infinite $\beta$ the right side is infinite.

Put $J(v)=\inf\{\sum_jV_{\beta_j}(u_j):\sum_jc_ju_j=v\}$.
Its sublevels are compact linear images, so $J$ is closed; mixing feasible points makes it convex. It is even, with $J(0)=0$, and
$J^*(z)=\sum_jf_{\beta_j}(c_jz)$.
Independence, \eqref{mgf}, Chernoff's bound and biconjugacy give
\[
\PP(H>v)\le\inf_{z\ge0}\exp(-zv+J^*(8z))=e^{-J(v/8)}.
\]
Formula \eqref{support} implies $J(2Q_\Phi(p))\ge p$ (or this point lies outside its domain). By symmetry,
$\PP(|H|>16Q_\Phi(p))\le2e^{-p}$.
The support profile is concave and vanishes at zero, so
$Q_\Phi(vp)\le vQ_\Phi(p)$ for $v\ge1$. Integrating the bound at budget $vp$ yields
$\|H\|_p\le C Q_\Phi(p)$ for $p\ge2$.

For the lower bound take nonnegative feasible $u_j$ in \eqref{support}. Coordinates with $u_j\le1$ contribute at most three. On $I=\{j:u_j>1\}$, the budget gives $|I|\le p$ and $\sum_{j\in I}u_j\le p$. The events
$\xi_j^{\beta_j}\ge u_j$ have probabilities $\tfrac12e^{-u_j}$, including equality at a finite cap. Force all these events and require the independent symmetric remainder to be nonnegative. The intersection has probability at least
$\tfrac12e^{-(1+\log2)p}$, so
$\|H\|_p\ge c\sum_{j\in I}c_ju_j$.
The core also gives $\E|H|\ge\E|\xi^1|=1-e^{-1}$ by conditional Jensen. Combining the two bounds absorbs the contribution at most three, and taking the supremum proves $Q_\Phi(p)\le C\|H\|_p$.

For the full expansion, truncate at bin $J$. The omitted coefficients have sum tending to zero, so the sources converge in every fixed $L_p$. The partial costs increase, and monotone convergence in the defining integral gives convergence of their profiles. Thus $\|H\|_p\asymp Q_\Phi(p)\asymp q(p)$ for $p\ge2$, with the same positive first-moment lower bound. Lemma~\ref{momentcx} gives the scalar convex orders. Their independent product and \eqref{cxpair} give the asserted pairing comparison.
\end{proof}

\section{Completion of the general theorem}\label{finish}
For a finite expansion, let $H=L^{\mathsf T}\Xi$. The pairing identity is
\begin{equation}\label{pullpair}
B_\Xi(L_\#\mu)=B_H(\mu).
\end{equation}
One direction follows by setting $X=L^{-1}U$ for $U\sim L_\#\mu$ and using
$\ip{U}{\Xi}=\ip{X}{L^{\mathsf T}\Xi}$.
For the other, take any coupling of $X,H$ and draw $\Xi$ from its conditional law given $L^{\mathsf T}\Xi=H$. This retains the prescribed law of $\Xi$ and exactly the same pairing.
The induced distortion is precisely $\Phi$, so the injective pullback rule and Theorem~\ref{capped} yield
\[
\RR_{\Phi^{(J)}}(\mu)=\RR_{\rm cap}((L_J)_\#\mu)
\asymp B_{\Xi^{(J)}}((L_J)_\#\mu)=B_{H^{(J)}}(\mu)
\]
for every finite truncation. The sources converge in $L_1$ and the cost matrices increase on the same finite label space. The limit facts in Section~\ref{prelim} therefore give
$\RR_\Phi(\mu)\asymp B_H(\mu)$ with the same constants.
Propositions~\ref{costprop} and \ref{sourceprop} now prove
\[
\boxed{\RR_\Psi(\mu)\asymp\RR_\Phi(\mu)\asymp B_H(\mu)\asymp B_Y(\mu).}
\]
This proves Theorem~\ref{main}; \eqref{supmu} gives the expected-supremum formulation.

\section{Special cases and scope}
For symmetric signs, $b_i=1$ and $\Psi_i(r)=\min(r^2/4,1)$.
For normalized symmetric exponentials, $N_i(u)=u$ and the slope is $2u$ below one and $2$ above one; thus $\Psi_i=\psi_\infty$.
For tails $N_i(u)=u^2$, the potential is quadratic up to universal factors, and so is its intrinsic cost, recovering the Gaussian-type geometry. Bounded coordinates and tails with arbitrarily rapidly increasing slopes are covered without a doubling constant.

The theorem proved here concerns independent coordinates. It is a prescribed-law majorizing-measure comparison, not a claim about arbitrary negatively associated log-concave vectors. The weighted-Orlicz SMP argument from the discussion is a separate result and is not used in this proof. Nor does the present manuscript derive the traditional $\ell_1+\gamma_2$ decomposition from the Bernoulli prescribed-law formula; that would require the corresponding geometric conversion.

\paragraph{Attribution and dependence.}
The capped variational route and its replication are taken from the supplied invariant draft. The clipped-code construction, slope-area identity, dyadic cost comparison and source comparison originate in \cite{HWW}; their proofs are included here to close the gap left explicit in the supplied generalization note. The only non-elementary probabilistic inputs retained are the standard exponential chaining theorem, Pr\'ekopa's theorem and the martingale characterization of convex order. In particular, the proof does not use the general uniform truncation theorem from \cite{HWW}, and the general-tail conclusion is not a new theorem relative to that work.

\end{document}